\documentclass[reqno]{amsart}
\usepackage{times}

\usepackage{amsmath}
\usepackage{amsthm}
\usepackage{amsfonts}
\usepackage{graphicx,psfrag,epsfig}
\usepackage{color}

\newtheorem{thm}{Theorem}[section]
\newtheorem{lem}[thm]{Lemma}
\newtheorem{cor}[thm]{Corollary}
\newtheorem{prop}[thm]{Proposition}

\DeclareMathAlphabet{\mathpzc}{OT1}{pzc}{m}{it}

\numberwithin{equation}{section}

\newcommand{\R}{\mathbb{R}}

\newcommand{\e}{\varepsilon}
\newcommand{\rd}{\mathrm{d}}

\newcommand{\ue}{u_\varepsilon}

\newcommand{\Phie}{\Phi_\varepsilon}

\newcommand{\dhr}{\mathrel{\lhook\joinrel\relbar\kern-.8ex\joinrel\lhook\joinrel\rightarrow}} 

\title[A stationary free boundary problem for MEMS]
{A stationary free boundary problem modeling electrostatic MEMS}

\begin{document}

\author{Philippe Lauren\c{c}ot}
\address{Institut de Math\'ematiques de Toulouse, CNRS UMR~5219, Universit\'e de Toulouse \\ F--31062 Toulouse Cedex 9, France}
\email{laurenco@math.univ-toulouse.fr}

\author{Christoph Walker}
\address{Leibniz Universit\"at Hannover\\ Institut f\" ur Angewandte Mathematik \\ Welfengarten 1 \\ D--30167 Hannover\\ Germany}
\email{walker@ifam.uni-hannover.de}

\begin{abstract}
A free boundary problem describing small deformations in a membrane based model of electrostatically actuated MEMS is investigated. The existence of stationary solutions is established for small voltage values. A justification of the widely studied narrow-gap model is given by showing that steady state solutions of the free boundary problem converge toward stationary solutions of the narrow-gap model when the aspect ratio of the device tends to zero.
\end{abstract}

\keywords{MEMS, free boundary problem, small-aspect ratio limit}
\subjclass[2010]{35R35, 35J57, 35B30, 74F15}

\maketitle

\section{Introduction}

Microelectromechanical systems (MEMS) have become key components of many commercial systems, including accelerometers for airbag deployment in automobiles, ink jet printer heads, optical switches, micropumps, chemical sensors and many others. Idealized modern MEMS devices often consist of two components: a rigid ground plate and a thin and deformable elastic membrane that is held fixed along its boundary above the rigid plate, and its design is based on the interaction of electrostatic and elastic forces. More precisely, when a voltage difference is applied between the two components, a Coulomb force is induced which is varied in strength by varying the applied voltage and gives rise to deformations of the elastic membrane. Perhaps the most ubiquitous nonlinear phenomenon associated with electrostatically actuated MEMS devices is the so-called ``pull-in'' instability limiting the effectiveness of such devices. In this instability, when voltages are applied beyond a certain critical pull-in voltage, there is no longer a steady-state configuration of the device where the two components remain separate. This possible touchdown of the membrane on the ground plate affects the design of the devices as it severely restricts the range of stable operation. The understanding and control of the pull-in voltage instability are thus of great technological importance: in this connection, a large number of MEMS devices which rely on electrostatic actuation have been investigated both experimentally and through numerical simulations and several mathematical models describing these devices have been set up. 

\begin{figure}
\centering\includegraphics[width=10cm]{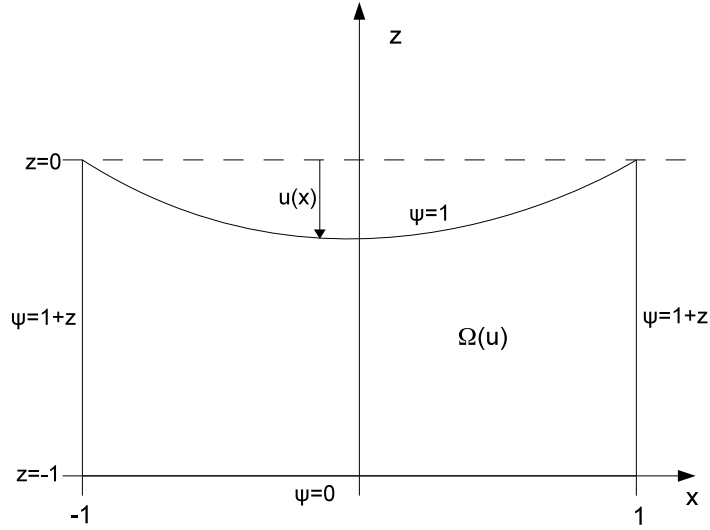}
\caption{\small Idealized electrostatic MEMS device.}\label{fig1}
\end{figure}

{We consider here a simple membrane based model of an electrostatically actuated MEMS device as depicted in Figure~\ref{fig1} and refer the reader e.g. to \cite{PeleskoSIAP02, BernsteinPelesko, PeleskoTriolo01} and the references therein for a more detailed account of the physical} background and the modeling aspects of modern MEMS devices. In this simplified situation, we assume that the applied voltage and the permittivity of the membrane are constant (normalized to one) and that there is no variation in the horizontal direction orthogonal to the $x$-direction of both the (dimensionless) electrostatic potential $\psi$ and the displacement $u$ of the membrane. Under appropriate scalings, the rigid ground plate is at $z=-1$ and the undeflected membrane at $z=0$ is fixed at the boundary $x=-1$ and $x=1$, see Figure~\ref{fig1}. Denoting the aspect ratio of the device, i.e. the ratio of the undeformed gap size to the device length, before scaling by $\varepsilon$, the dimensionless electrostatic potential $\psi=\psi(x,z)$ is supposed to satisfy Laplace's equation
\begin{equation}\label{psi}
\varepsilon^2\partial_x^2\psi + \partial_z^2\psi =0
\end{equation}
in the region $$\Omega(u) := \left\{ (x,z)\in (-1,1)\times (-1,\infty)\ :\ -1 < z < u(x) \right\}$$ between the rigid ground plate at $z=-1$ and the deflected membrane at $z=u$. The boundary conditions are then
\begin{equation}\label{a}
\psi=0\quad \text{on}\quad z=-1
\end{equation}
and
\begin{equation}\label{b}
\psi=1\quad \text{on}\quad z=u\ .
\end{equation}
As for the deformation of the membrane, it results from a balance between dynamic, electrostatic, and elastic forces and the membrane displacement $u=u(t,x)\in (-1,\infty)$ evolves according to Newton's law
\begin{equation}\label{u}
\alpha^2\partial_t^2 u + \partial_t u - \partial_x^2 u = - \lambda\ \left( \varepsilon^2\ |\partial_x\psi(x,u)|^2 + |\partial_z\psi(x,u)|^2 \right)
\end{equation}
with {clamped} boundary conditions
\begin{equation}\label{bcu}
u=0\quad\text{at}\quad x=\pm 1\ .
\end{equation}
In \eqref{u}, the term $\partial_t u$ and the right-hand side account for a damping force and the electrostatic force, respectively, while the term $\partial_x^2 u$ describes the deformation due to stretching. The latter is obtained after linearization resulting from the assumption of small deformations. The contribution to deformation due to bending may also be included in \eqref{u} by adding a fourth-order term $B \partial_x^4 u$, $B>0$, to the left-hand side of \eqref{u} but is neglected here. The parameter $\lambda\ge 0$ characterizes the relative strengths of electrostatic and mechanical forces. It acts as a control parameter proportional to the applied voltage. The coefficient $\alpha\ge 0$ is indirectly proportional to the damping coefficient.

Observe that \eqref{psi} is a free boundary problem as the domain between the rigid ground plate and the elastic membrane changes with time. Due to this, equations \eqref{psi} and \eqref{u} are strongly coupled. However, a common assumption made in all mathematical analysis hitherto is a small aspect ratio $\varepsilon$. Formally, sending $\varepsilon$ to zero allows one to solve explicitly \eqref{psi}-\eqref{b} for the potential $\psi=\psi_0$, i.e.
\begin{equation}\label{sar}
\psi_0(x,z)=\frac{1+z}{1+u(x)}\ , \quad (x,z)\in \Omega(u)\ ,
\end{equation}
thus reducing the free boundary problem to the small aspect ratio model, that is, to an evolution equation  
\begin{equation}\label{usar}
\alpha^2\partial_t^2 u+\partial_t u-\partial_x^2 u = -\frac{\lambda}{(1+u)^2}\,, \quad {(t,x)\in (0,\infty)\times (-1,1)\,,}
\end{equation}
subject to \eqref{bcu} solely involving the displacement $u$.
Note that the boundary conditions \eqref{a}, \eqref{b} suffice to determine $\psi_0$ which then satisfies on the lateral boundaries $x=\pm 1$
\begin{equation}\label{l}
\psi_0(\pm 1,z)=1+z\,, \quad {z\in (-1,0)\,,}
\end{equation}
due to \eqref{bcu}.
The small aspect ratio model \eqref{usar} with \eqref{bcu} has widely been investigated in the recent past (with possibly an additional fourth-order term accounting for the deformation due to bending as already discussed) and also variants thereof, e.g. in higher dimensions or with additional permittivity profiles or non-local terms. An obvious difficulty arising in the study of \eqref{usar} is the singularity of the source term $-\lambda/(1+u)^2$ as $u$ approaches $-1$ which corresponds to the aforementioned touchdown phenomenon for the MEMS device. Concerning the dynamic behavior of small aspect ratio models we refer the reader to \cite{FloresMercadoPelesko, GuoSIAMDynamic10, KavallarisLaceyNikolopoulos} for the hyperbolic case $\alpha>0$ and to \cite{FloresMercadoPeleskoSmyth_SIAP07, GhoussoubGuoNoDEA08, GuoJDE08, GuoJDE08_II, Hui11, PeleskoSIAP02} for the corresponding parabolic equation with $\alpha=0$ when damping or viscous forces dominate over inertial forces in the system. The dynamic behavior of a membrane evolving according to \eqref{usar}, \eqref{bcu} is determined by a pull-in voltage $\lambda_*>0$, see e.g. \cite{BernsteinGuidottiPelesko, GhoussoubGuoSIMA07, LinYang, PeleskoSIAP02}. More precisely, if $\lambda<\lambda_*$ there are a stable and an unstable steady state (i.e. time independent) solution of \eqref{usar} subject to \eqref{bcu}, that is, of
\begin{equation}\label{uu}
\partial_x^2 u=\frac{\lambda}{(1+u)^2}\ ,\quad x\in (-1,1)\ ,\qquad u(\pm 1)=0\ ,
\end{equation}
and solutions to the dynamical problem \eqref{usar} starting out from $u=0$ converge toward the stable steady state. Steady states cease to exist for voltage values $\lambda$ above $\lambda_*$  and solutions to the dynamic problem touch down on the ground plate in finite time, that is, a pull-in instability occurs. We refer the reader to \cite{EspositoGhoussoubGuo, FloresMercadoPelesko, BernsteinPelesko} for a review of these results and references as well as to \cite{BernsteinGuidottiPelesko, FloresMercadoPeleskoSmyth_SIAP07, GhoussoubGuoSIMA07, GhoussoubGuoNoDEA08, GuoJDE08, GuoJDE08_II, PeleskoSIAP02, PeleskoDriscoll} and the references therein for further details on small aspect ratio models. 

\medskip

To the best of our knowledge, the original model without small gap assumption has not been tackled so far from an analytical point of view. The aim of this paper is to make a step in this direction by studying the stationary free boundary problem. More precisely, we shall focus on finding functions \mbox{$u:[-1,1]\rightarrow (-1,{\infty)}$} and $\psi:\overline{\Omega(u)}\rightarrow \R$ satisfying the coupled system of elliptic equations
\begin{align}
\varepsilon^2\partial_x^2\psi(x,z) + \partial_z^2\psi(x,z) &=0\ ,\quad  & (x,z)\in \Omega(u)\ ,\label{1}\\
\psi(x,z)&=1+z\ , & (x,z)\in \partial\Omega(u)\ ,\label{2}\\
\partial_x^2 u(x) &=\lambda \big(\varepsilon^2\vert\partial_x\psi(x,u(x))\vert^2+\vert\partial_z\psi(x,u(x))\vert^2\big)\ , &x\in (-1,1)\ ,\label{3}\\
u(x)&=0\ , & x=\pm 1\ ,\label{4}
\end{align}
where the domain of definition $\overline{\Omega(u)}$ of $\psi$ is
$$
\Omega(u):=\big\{(x,z)\,;\, -1<x<1\, ,\, -1<z<u(x)\big\} \,,
$$
that is, the two-dimensional region between the rigid ground plate and the membrane with deflection $u$. Obviously, $\Omega(u)$ is a domain and possesses four corners provided the values of the continuous and convex (see \eqref{3}) function $u$ satisfying \eqref{4} stay away from $-1$.
Let then $$\Gamma(u):=\partial\Omega(u)\setminus\{(\pm 1,-1), (\pm 1,0)\}$$ denote the boundary of $\Omega(u)$ without corners. System \eqref{1}-\eqref{4} is exactly the time-independent version of \eqref{psi}-\eqref{bcu} subject to the lateral boundary condition~\eqref{l}  which is imposed to make the system well-posed. This particular choice of a continuous boundary condition is made mainly for simplicity and we point out again that this condition is satisfied by $\psi_0$ from \eqref{sar} in the small aspect ratio limit. For the stationary free boundary problem we shall show existence of smooth solutions for small voltage values $\lambda$:

\begin{thm}\label{A}
There exists $\lambda_0>0$  independent of $\e\in (0,1)$ such that \eqref{1}-\eqref{4} admits for each $\lambda\in (0,\lambda_0]$ a solution
$$
u\in C^{2+\alpha}\big([-1,1]\big)\ ,\quad \psi\in W_2^2\big(\Omega(u)\big) \cap C\big(\overline{\Omega(u)}\big)\cap C^{2+\alpha}\big(\Omega(u)\cup\Gamma(u)\big)\ ,
$$
where $\alpha\in [0,1)$ is arbitrary. The function $u$ is even, convex, and satisfies
\begin{align}
&0\ge u(x)\ge -1+\kappa_0\ ,\quad x\in (-1,1)\ ,\label{41}\\
& \| u\|_{W_\infty^2(-1,1)}\le \frac{1}{\kappa_0}\ ,\label{42}
\end{align}
for some $\kappa_0\in (0,1)$ independent of $\varepsilon$. Moreover, $\psi=\psi(x,z)$ is even with respect to $x\in (-1,1)$.
\end{thm}

We refer to Section~\ref{Sec2} for the proof of Theorem~\ref{A} which is based on a transformation to a fixed domain and on an application of Schauder's fixed point theorem. Concerning the latter, given a displacement $u\in W_\infty^2(-1,1)$ with values in $(-1,0)$, we first construct in Lemma~\ref{L32} the corresponding solution $\psi_u$ to \eqref{1}-\eqref{2} by using an equivalent formulation on a rectangle. Of particular importance is the regularity of the trace of the gradient of $\psi_u$ on the upper boundary $\{z=u\}$ as stated in Lemma~\ref{L34} which plays an important role in the subsequent analysis of \eqref{3}-\eqref{4}. Indeed, it is used as a source term to construct a solution $S(u)$ to \eqref{3}-\eqref{4} with $\psi_u$ instead of $\psi$, see Lemma~\ref{L34a}. Restricting suitably the set of admissible displacements $u$, the map $S$ turns out to enjoy the properties needed to apply Schauder's fixed point theorem.

\medskip

In particular, for values $\lambda\le\lambda_0$, Theorem~\ref{A} provides for each $\varepsilon\in (0,1)$ a solution $(u_\varepsilon,\psi_\varepsilon)$ to \eqref{1}-\eqref{4} satisfying the bounds \eqref{41}, \eqref{42} uniformly with respect to $\e\in (0,1)$. This property allows us to give a rigorous justification of the small aspect ratio model \eqref{usar}, \eqref{bcu} by showing that $(u_\varepsilon,\psi_\varepsilon)_{\varepsilon\in (0,1)}$ converges toward a solution to that model as $\varepsilon$ tends to zero. More generally, we have:

\begin{thm}\label{B}
Let $\lambda>0$ and let $(u_\varepsilon,\psi_\varepsilon)_{\varepsilon\in (0,1)}$ be a family of solutions to \eqref{1}-\eqref{4} satisfying the bounds \eqref{41} and \eqref{42}. Then there are a sequence $\varepsilon_k\searrow 0$ and a (smooth) solution $u_0$ to the time-independent small aspect ratio equation \eqref{uu}
such that
$$ 
u_{\varepsilon_k}\longrightarrow u_0\quad \text{in}\quad W_\infty^1(-1,1)
$$
and
\begin{equation}\label{411} 
\psi_{\varepsilon_k}\mathbf{1}_{\Omega(u_{\varepsilon_k})}\longrightarrow \psi_{0}\mathbf{1}_{\Omega(u_{0})}\quad \text{in}\quad L_2\big((-1,1)\times (0,1)\big)
\end{equation}
as $k\rightarrow \infty$, where $\psi_0$ is the corresponding potential \eqref{sar} with $u=u_0$.
\end{thm}

The proof of Theorem~\ref{B} is performed in Section~\ref{Sec3} by using a compactness argument. Since \eqref{1} becomes degenerate elliptic in the limit $\varepsilon\to 0$, the regularity of $\psi_\varepsilon$ is no longer the same in the $x$- and $z$-directions and a cornerstone of the proof is to obtain estimates for the trace of $\partial_z \psi_\e$ on $\{ z = u_\e\}$.

\section{Existence for small voltage values: Proof of Theorem~\ref{A}}\label{Sec2} 

We first prove Theorem~\ref{A}. Since the domain of definition of the potential $\psi$ in \eqref{1} depends on the displacement $u$ of the membrane, we use an alternative formulation by transforming the problem on a fixed domain, that is, on the rectangle $\Omega:=(-1,1)\times (0,1)$. {More precisely}, given a function $u\in W_\infty^2(-1,1)$ taking values in $(-1,{\infty)}$ and satisfying the boundary conditions $u(\pm 1)=0$, we define a diffeomorphism \mbox{$T_u:=\overline{\Omega(u)}\rightarrow \bar{\Omega}$} by setting
\begin{equation}\label{Tu}
T_u(x,z):=\left(x,\frac{1+z}{1+u(x)}\right)\ ,\quad (x,z)\in \overline{\Omega(u)}\ .
\end{equation}
Clearly,
\begin{equation}\label{Tuu}
T_u^{-1}(x,\eta)=\big(x,(1+u(x))\eta-1\big)\ ,\quad (x,\eta)\in \bar{\Omega}\ ,
\end{equation}
and it readily follows that problem \eqref{1}-\eqref{2} is equivalent to
\begin{eqnarray}
\big(\mathcal{L}_u\phi\big) (x,\eta)\!\!\!&=0\ ,&(x,\eta)\in\Omega\ ,\label{23}\\
\phi(x,\eta)\!\!\!&=\eta\ , &(x,\eta)\in \partial\Omega\ ,\label{24}
\end{eqnarray}
for $\phi=\psi\circ T_u^{-1}$, {the $u$-dependent differential operator $\mathcal{L}_u$ being defined by}
\begin{equation*}
\begin{split}
\mathcal{L}_u w\, :=\, & \e^2\partial_x^2 w-2\e^2\eta\frac{\partial_x u(x)}{1+u(x)}\partial_x\partial_\eta w
+\frac{1+\e^2\eta^2(\partial_x u(x))^2}{(1+u(x))^2}\partial_\eta^2 w\\
& +\e^2\eta\left[2\left(\frac{\partial_x u(x)}{1+u(x)}\right)^2-\frac{\partial_x^2 u(x)}{1+u(x)}\right] \partial_\eta w\ .
\end{split}
\end{equation*}
Moreover,  \eqref{3}, \eqref{4} become
\begin{align}
\partial_x^2 u(x) &= \lambda\left[\frac{1+\e^2(\partial_x u(x))^2}{(1+u(x))^2} \right]\vert\partial_\eta\phi(x,1)\vert^2\ , & x\in (-1,1)\ ,\label{33}\\
u(x) &=0\ , & x=\pm 1\ ,\label{34}
\end{align}
where we have used
\begin{equation}\label{26}
\partial_x\phi(x,1)=0\ ,\quad x\in (-1,1)\ ,
\end{equation}
since $\phi(x,1)=1$ for $x\in (-1,1)$ by {\eqref{24}}.

Our goal is to solve \eqref{23}-\eqref{34} by means of Schauder's fixed point theorem. Fixing $r_0\in (0,2)$, we introduce the set
$$
\mathcal{C}:=\big\{u\in W_\infty^2(-1,1)\cap W_{2,D}^2(-1,1)\ :\ {u\ \text{ is even and }\ 0 \le \partial_x^2 u \le r_0 }\big\} \ \,,
$$
where
$$
W_{q,D}^2(-1,1):=\big\{u\in W_{q}^2(-1,1)\ :\ u(\pm 1)=0\big\} \ , \qquad {q\in [1,\infty]}\,.
$$
{Let us first collect some properties of $\mathcal{C}$.}

\begin{lem}\label{L31}
$\mathcal{C}$ is a closed, convex, and bounded subset of $W_q^2(-1,1)$ for each $q\in [1,\infty]$ and
\begin{equation}\label{36}
0\ge u(x)\ge-\frac{r_0}{2}>-1\ ,\quad x\in (-1,1)\ ,\quad u\in \mathcal{C}\ .
\end{equation}
\end{lem}

\begin{proof}
{Clearly}, $\mathcal{C}$ is convex and closed in $ W_\infty^2(-1,1)$ and thus weakly closed  in $ W_\infty^2(-1,1)$. {Therefore,} $\mathcal{C}$ is convex and closed in $W_q^2(-1,1)$ for each $q\in [1,\infty]$. Next, for $ u\in \mathcal{C}$, {integrating the equality}
$$
\partial_x u(x)=\partial_x u(y)+\int_y^x\partial_x^2 u(z)\,\rd z\ ,\quad {(x,y)\in (-1,1)\times (-1,1)}\ ,
$$
with respect to $y$ on $(-1,1)$, {we find:}
\begin{equation}\label{36c}
\vert\partial_x u(x)\vert\le 2 r_0\ ,\quad x\in (-1,1)\ .
\end{equation}
{Since $u(\pm 1)=0$, we deduce from \eqref{36c} that $|u(x)|\le 2 r_0$ for $x\in [-1,1]$ and $\mathcal{C}$ is thus} bounded in $W_\infty^2(-1,1)$. {Next, the convexity of $u$ and the boundary values $u(\pm 1)=0$ clearly ensure that $u\le 0$. Finally, if} $u$ attains a negative minimum at some point $x_m\in (-1,1)$, we may assume $x_m\in [0,1)$ without loss of generality since $u$ is even. Then $\partial_x u(x_m)=0$ and
$$
u(x)-u(x_m)=\int_{x_m}^x\partial_x u(y)\,\rd y=\big[(y-x)\partial_x u(y)\big]_{y=x_m}^{y=x}-\int_{x_m}^x(y-x)\partial_x^2 u(y)\,\rd y=\int_{x_m}^x(x-y)\partial_x^2 u(y)\,\rd y\ .
$$
Thus, since $u(1)=0$,
$$
-u(x_m)=\int_{x_m}^1 (1-y)\partial_x^2 u(y)\,\rd y\le \frac{r_0}{2}\ ,
$$
{from which} \eqref{36} follows. 
\end{proof}

Next we study the existence and properties of the solution to \eqref{23}-\eqref{24} when $u\in\mathcal{C}$ is given.

\begin{lem}\label{L32}
Given $u\in\mathcal{C}$ there is a unique solution $\phi_u\in W_{2}^2(\Omega)$ to \eqref{23}-\eqref{24}. Moreover, $\phi_u=\phi_u(x,\eta)$ is even with respect to $x$,
\begin{equation}\label{266}
\eta\big(1+u(x)\big)\le\phi_u(x,\eta)\le 1\ ,\quad (x,\eta)\in\bar{\Omega}\ ,\quad u\in\mathcal{C}\ ,
\end{equation}
and 
\begin{equation}\label{36a}
\|\phi_u\|_{W_2^2(\Omega)}\le c_1\ ,\quad u\in\mathcal{C}\ ,
\end{equation}
for some constant $c_1=c_1(r_0,\e)>0$.
\end{lem}

\begin{proof}
We claim that the operator $-\mathcal{L}_u$ is elliptic for $u\in\mathcal{C}$ given. To see this, choose an arbitrary $u\in\mathcal{C}$ and let
$$
A:=\left(\begin{matrix} 
 \e^2 & & \displaystyle{-\frac{\e^2\partial_x u(x)}{1+u(x)}\eta} \\ 
 & & \\
\displaystyle{-\frac{\e^2\partial_x u(x)}{1+u(x)}\eta} & & \displaystyle{\frac{1+\e^2\eta^2(\partial_x u(x))^2}{(1+u(x))^2}}
\end{matrix}\right)
$$
denote the principal part of $-\mathcal{L}_u$ for fixed $(x,\eta)\in\bar{\Omega}$ with trace $t$ and determinant
$ d $ given by
$$
t:= \e^2 + \frac{1+\e^2\eta^2(\partial_x u(x))^2}{(1+u(x))^2}\ ,\qquad d:= \frac{\e^2}{(1+u(x))^2}\ .
$$
Then the two eigenvalues of $A$ are
$$
\mu_\pm =\frac{1}{2}\big( t\pm\sqrt{t^2-4d}\big)
$$
and since
$$
1+\e^2\le t \le \e^2+\frac{4}{(2-r_0)^2}\ \left( 1 + 4 \e^2 r_0^2 \right)\ ,\qquad d\ge \e^2\ ,
$$
by \eqref{36} and \eqref{36c}, 
$$
\mu_+ \ge \mu_- \ge \frac{d}{t} \ge \frac{\e^2 (2-r_0)^2}{\e^2 (2-r_0)^2 + 4 + 16 \e^2 r_0^2} > 0\,.
$$ 
Consequently, $-\mathcal{L}_u$ is elliptic with a positive ellipticity constant depending on $r_0$ and $\e$ but not on $u\in\mathcal{C}$. Next observe that \eqref{23}-\eqref{24} is equivalent to
\begin{align}
\big(\mathcal{L}_u\Phi\big) (x,\eta) &=-f_u(x,\eta)\ ,&(x,\eta)\in\Omega\ ,\label{23a}\\
\Phi(x,\eta) &=0\ , &(x,\eta)\in \partial\Omega\ ,\label{24a}
\end{align}
by setting $\Phi(x,\eta):=\phi(x,\eta)-\eta$, {$(x,\eta)\in \bar{\Omega}$}, where $f_u\in L_\infty(\Omega)$ is defined as
\begin{equation}\label{f}
f_u(x,\eta):=\mathcal{L}_u\eta =\e^2\eta\left[2\left(\frac{\partial_x u(x)}{1+u(x)}\right)^2-\frac{\partial_x^2 u(x)}{1+u(x)}\right]\ ,\quad (x,\eta)\in\Omega\ ,
\end{equation}
{and satisfies
\begin{equation}
\|f_u\|_{L_\infty(\Omega)} \le \e^2\ \left( \frac{32 r_0^2}{(2-r_0)^2} + \frac{2r_0}{2-r_0} \right) \label{spip}
\end{equation}
by \eqref{36c} and Lemma~\ref{L31}. Noticing that, thanks to Lemma~\ref{L31},} all coefficients of $\mathcal{L}_u$, written in divergence form 
\begin{equation*}
\begin{split}
\mathcal{L}_u w\, =\, & \partial_x\left(\e^2\partial_x w-\e^2\eta\frac{\partial_x u(x)}{1+u(x)}\partial_\eta w\right)
+\partial_\eta\left(-\e^2\eta\frac{\partial_x u(x)}{1+u(x)}\partial_x w +\frac{1+\e^2\eta^2(\partial_x u(x))^2}{(1+u(x))^2}\partial_\eta w\right)\\
& +\e^2\frac{\partial_x u(x)}{1+u(x)}\partial_x w-\e^2\eta\left(\frac{\partial_x u(x)}{1+u(x)}\right)^2\partial_\eta w\ ,
\end{split}
\end{equation*}
as well as $f_u$ have norms in $L_\infty(\Omega)$ uniformly bounded with respect to $u\in\mathcal{C}$, 
we may apply \cite[Chapt.~3, Thm.~9.1 $\&$ Thm.~10.1]{LadyzhenskayaUraltsevaLQEE} to obtain the existence and uniqueness of a solution $\Phi_u\in  W_{2,D}^2(\Omega)$ to \eqref{23a}-\eqref{24a} satisfying
\begin{equation}
\|\Phi_u\|_{W_2^2(\Omega)}\le c \big(\|\Phi_u\|_{L_2(\Omega)} +1\big) \label{gaston}
\end{equation}
with a constant $c$ depending on $r_0$ and $\e$ but not on $u\in\mathcal{C}$. Setting $\phi_u(x,\eta)=\Phi_u(x,\eta)+\eta$ for $(x,\eta)\in\bar{\Omega}$, the function $\phi_u$ obviously solves \eqref{23}-\eqref{24} and, owing to \eqref{gaston}, the bound \eqref{36a} readily follows provided we can verify \eqref{266}. For this we take $w\equiv 1$ and note that $\mathcal{L}_u w=0$ \mbox{in $\Omega$} while $w(\eta)=1\ge \eta=\phi_u(x,\eta)$ for $(x,\eta)\in\partial\Omega$. The comparison principle then ensures $\phi_u\le 1$ \mbox{in $\bar{\Omega}$}. Taking $v(x,\eta):=\eta (1+u(x))$ for $(x,\eta)\in\bar{\Omega}$, 
we have $\mathcal{L}_u v=0$ \mbox{in $\Omega$} and $v(x,\eta)\le \eta =\phi_u(x,\eta)$ for $(x,\eta)\in\partial\Omega$. We conclude that $\phi_u\ge v$ in $\bar{\Omega}$ again by the comparison principle, and \eqref{266} and \eqref{36a} follow. It remains to check that $\phi_u$ is even. However, {$u$ being} even, $\partial_x u$ is odd and $\partial_x^2 u$ is even, and it is easily seen that $\tilde{\phi}(x,\eta):=\phi_u(-x,\eta)$ satisfies \eqref{23}-\eqref{24} as well, whence $\tilde{\phi}=\phi_u$ by uniqueness.
\end{proof}

We next turn to the continuity property of $\phi_u$ with respect to $u\in\mathcal{C}$.

\begin{lem}\label{L33}
The mapping $\big(u\mapsto \phi_u\big):\mathcal{C}\longrightarrow W_{2}^2(\Omega)$ is continuous when $\mathcal{C}$ is endowed with the topology of $W_{2}^2(-1,1)$.
\end{lem}

\begin{proof}
Let $\mathcal{C}$ be endowed with the topology of $W_{2}^2(-1,1)$. Given $u\in\mathcal{C}$, we define a bounded linear operator $A(u)\in\mathcal{L}\big(W_{2,D}^2(\Omega), L_2(\Omega)\big)$ by
$$
A(u)\Phi:=-\mathcal{L}_u\Phi\ ,\quad \Phi\in W_{2,D}^2(\Omega)\ ,
$$
{and note that $A$ is continuous from $\mathcal{C}$ in $\mathcal{L}\left( W_{2,D}^2(\Omega),L_2(\Omega) \right)$, thanks to the continuous embedding of $W_2^2(-1,1)$ in $W_\infty^ 1(-1,1)$ and the boundedness of $\mathcal{C}$ in $W_\infty^2(-1,1)$.}
Then \cite[Chapt.~3, Thm.~9.1 $\&$ Thm.~10.1]{LadyzhenskayaUraltsevaLQEE} (see the proof of Lemma~\ref{L32}) warrants that $A(u)$ is invertible for each $u\in\mathcal{C}$. Owing to the continuity (in fact: analyticity) of the inversion map $\ell\mapsto\ell^{-1}$ of bounded linear operators, we conclude that 
$$
\mathcal{C}\longrightarrow \mathcal{L}\big(L_2(\Omega),W_{2,D}^2(\Omega)\big)\ ,\quad u\mapsto A(u)^{-1}
$$
is continuous. One then checks that $u\mapsto f_u$ is continuous from $\mathcal{C}$ to $L_2(\Omega)$, where $f_u$ is given in \eqref{f}. Consequently, 
$$
\big(u\mapsto \Phi_u=A(u)^{-1}(-f_u)\big):\mathcal{C}\longrightarrow W_{2,D}^2(\Omega) 
$$
is continuous. {Recalling that} $\phi_u(x,\eta)=\Phi_u(x,\eta)+\eta$ for $(x,\eta)\in\bar{\Omega}$ gives the claim.
\end{proof}

To obtain estimates on solutions to \eqref{3}-\eqref{4} we need estimates on the gradient of $\phi_u$ on the boundary $\eta=1$ as provided by the following lemma.

\begin{lem}\label{L34}
There is a constant $c_2>0$ depending only on $r_0\in (0,2)$ and $\e\in (0,1)$ such that, given $u\in\mathcal{C}$, the corresponding solution $\phi_u\in W_{2}^2(\Omega)$ to \eqref{23}-\eqref{24} satisfies
\begin{equation}\label{311}
\|\partial_\eta\phi_u(\cdot,1)\|_{W_2^{1/2}(-1,1)}\le c_2
\end{equation}
and
\begin{equation}\label{312}
0\le\partial_\eta\phi_u(x,1)\le 1+2\e^2\ ,\quad x\in (-1,1)\ .
\end{equation}
\end{lem}

\begin{proof}
According to \cite[Chapt.~2, Thm.~5.4]{Necas67} there is a positive constant $c$ depending only on $\Omega$ such that
$$
\|\partial_\eta\phi_u(\cdot,1)\|_{W_2^{1/2}(-1,1)}\le c\,  \|\phi_u\|_{W_2^{2}(\Omega)}
$$
from which \eqref{311} readily follows by \eqref{36a}. Next, set $w_\alpha(\eta):=\eta^{1+\alpha}$ for $\eta\in [0,1]$ and $\alpha>0$. Then $w_\alpha(\eta)\le \eta=\phi_u(x,\eta)$ for $(x,\eta)\in\partial\Omega$ and
\begin{equation*}
\begin{split}
\mathcal{L}_uw_\alpha &=\frac{1+\e^2\eta^2(\partial_x u)^2}{(1+u)^2}\alpha (1+\alpha)\eta^{\alpha-1}+\e^2(\alpha+1)\eta^{\alpha+1}\left[2\left(\frac{\partial_x u}{1+u}\right)^2-\frac{\partial_x^2 u}{1+u}\right]\\
&\ge \frac{\alpha(1+\alpha)\eta^{\alpha-1}}{(1+u)^2}-\e^2(\alpha+1)\eta^{\alpha+1}\frac{\partial_x^2 u}{1+u}\\
&\ge \frac{(1+\alpha)\eta^{\alpha-1}}{(1+u)^2}\left[\alpha-\e^2\eta^2 (1+u)\ \partial_x^2 u \right]\ge \frac{(1+\alpha)\eta^{\alpha-1}}{(1+u)^2}\left[\alpha-2\e^2\right]
\end{split}
\end{equation*}
in $\Omega$, where we used $u\le 0$ and $0\le \partial_x^2u \le r_0<2$ to obtain the last inequality. Consequently, choosing $\alpha = 2\e^2$, we realize that $\mathcal{L}_uw_{2\e^2}\ge 0$ in~$\Omega$, and we infer from the comparison principle that 
$$
\phi_u(x,\eta)\ge w_{2\e^2}(\eta)\ ,\quad (x,\eta)\in\bar{\Omega}\ .
$$
In particular, for $\eta\in (0,1)$,
$$
\frac{1}{\eta-1}\left(\phi_u(x,\eta)-\phi_u(x,1)\right)=\frac{1}{\eta-1}\left(\phi_u(x,\eta)-1\right)\le \frac{1}{\eta-1}\left(w_{2\e^2}(\eta)-w_{2\e^2}(1)\right)\ ,
$$
whence $\partial_\eta\phi_u(x,1)\le {\partial_\eta w_{2\e^2}(1) =} 1+2\e^2$ for $x\in (-1,1)$. Since $\phi_u\le 1$ in $\Omega$ and $\phi_u(x,1)=1$, we also have $\partial_\eta\phi_u(x,1)\ge 0$ for $x\in (-1,1)$.
\end{proof}

Next, given $u\in\mathcal{C}$, we set
\begin{equation}\label{g}
g_u(x):= \frac{1+\e^2(\partial_x u(x))^2}{(1+u(x))^2} \, \vert\partial_\eta\phi_u(x,1)\vert^2\ ,\quad  x\in (-1,1)\ ,
\end{equation}
and observe that Lemma~\ref{L31} and \eqref{312} guarantee that $g_u\in L_\infty(-1,1)$. Thus, for each $\lambda>0$ there is a unique solution $v=S(u)$ in $W_{\infty,D}^2(-1,1)$ to the linear problem
\begin{align}
\partial_x^2 v(x) &= \lambda g_u(x)\ , & x\in (-1,1)\ ,\label{313}\\
v(x) &=0\ , & x=\pm 1\ .\label{314}
\end{align}
Actually, we have:

\begin{lem}\label{L34a}
If $\mathcal{C}$ is endowed with the topology of $W_{2}^2(-1,1)$, then $S:\mathcal{C}\rightarrow W_{2}^{2+\sigma}(-1,1)$ is continuous for each $\sigma\in [0,1/2)$, and there is a positive constant $c_3(\sigma)$ depending only $r_0$, $\e$, and $\sigma$ such that
\begin{equation}\label{315}
\|S(u)\|_{W_2^{2+\sigma}(-1,1)}\le \lambda\ c_3(\sigma)\ ,\quad  u\in\mathcal{C}\ .
\end{equation}
Moreover, $S(u)$ is even and convex for $u\in\mathcal{C}$.
\end{lem}

\begin{proof} 
Lemma~\ref{L33} together with \cite[Chapt.~{2}, Thm.~5.4]{Necas67} and \eqref{311} imply that $u\mapsto \partial_\eta\phi_u(\cdot,1)$ is continuous and bounded as a mapping $\mathcal{C}\rightarrow W_{2}^{1/2}(-1,1)$. 
In the following, given two Banach spaces  $X$ and $Y$ of real-valued functions, we write $X \hookrightarrow Y$ to indicate that $X$ is continuously embedded in $Y$ and we set $X\cdot Y := \{ fg\ : \ (f,g)\in X\times Y\}$. Let $0<\sigma<\sigma_1<1/2$.
Since pointwise multiplication 
$$
W_{2}^{1/2}(-1,1)\cdot W_{2}^{1/2}(-1,1) \hookrightarrow W_{2}^{\sigma_1}(-1,1)
$$ 
is bilinear and continuous according to \cite[Thm.~4.1 \& Rem.~4.2(d)]{AmannMultiplication}, we infer that $u\mapsto \vert\partial_\eta\phi_u(\cdot,1)\vert^2$ defines a bounded and continuous mapping $\mathcal{C}\longrightarrow  W_{2}^{\sigma_1}(-1,1)$. Noticing that 
\begin{equation}\label{cc}
\mathcal{C}\longrightarrow W_2^1(-1,1)\ ,\quad u\mapsto  \frac{1+\e^2(\partial_x u(x))^2}{(1+u(x))^2}
\end{equation}
is continuous and bounded as well by Lemma~\ref{L31} and that
$$
W_{2}^{1}(-1,1)\cdot W_{2}^{\sigma_1}(-1,1) \hookrightarrow W_{2}^{\sigma}(-1,1)
$$ 
by \cite[Thm.~4.1 \& Rem.~4.2(d)]{AmannMultiplication}, we see that $u\mapsto g_u$ is continuous and bounded from $\mathcal{C}$ to $W_{2}^{\sigma}(-1,1)$. Consequently, $S:\mathcal{C}\longrightarrow W_{2}^{2+\sigma}(-1,1)$ is continuous and satisfies \eqref{315}.
Clearly, $S(u)$ is even for~$u\in\mathcal{C}$ since $u$ and $\phi_u(\cdot,1)$ are even by Lemma~\ref{L32}, and $S(u)$ is convex since $\partial_x^2 S(u)\ge 0$ by \eqref{g} and \eqref{313}.
\end{proof}

We are now in a position to construct solutions to \eqref{23}-\eqref{34} for small values of $\lambda$ by applying  Schauder's fixed point theorem to the map $S$.

\begin{prop}\label{P1}
There exists $\lambda_0>0$ independent of $\e\in (0,1)$ such that \eqref{23}-\eqref{34} admits for each $\lambda\in (0,\lambda_0]$ a solution
$$
(u,\phi_u)\in W_\infty^2(-1,1) \times W_2^2(\Omega)
$$
satisfying
$$
0 \ge u(x)\ge -\frac{r_0}{2}>-1 \quad\text{ and }\quad 0 \le \partial_x^2 u(x) \le r_0\ ,\quad x\in (-1,1)\ .
$$
Moreover, $u$ is even and belongs to $W_2^{2+\sigma}(-1,1)$ for any $\sigma\in [0,1/2)$.
\end{prop}

\begin{proof}
To prove that $S$ maps the closed and convex subset $\mathcal{C}$ of $W_2^2(-1,1)$ into itself note that \eqref{36c}, \eqref{312} together with Lemma~\ref{L31} ensure
\begin{equation}\label{r0}
0\le\partial_x^2 S(u) =\lambda \frac{1+\e^2(\partial_x u(x))^2}{(1+u(x))^2} |\partial_\eta\phi_u(x,1)|^2\le 4\lambda\frac{1+4\e^2 r_0^2}{(2-r_0)^2} (1+2\e^2)
\end{equation}
for $u\in\mathcal{C}$. Thus there is $\lambda_0=\lambda_0(r_0)>0$ sufficiently small and  independent of $\e\in (0,1)$ such that $0\le \partial_x^2 S(u)\le r_0$ for $\lambda\in (0,\lambda_0]$ and $u\in\mathcal{C}$, so it follows from Lemma~\ref{L34a}
that $S$ indeed maps $\mathcal{C}$ into itself. Since $W_2^{2+\sigma}(-1,1)$ embeds compactly in $W_2^{2}(-1,1)$ for $\sigma\in (0,1/2)$, Lemma~\ref{L34a} implies that $S:\mathcal{C}\rightarrow \mathcal{C}$ is continuous and compact and thus has a fixed point $ u\in\mathcal{C}$ enjoying the properties stated in Lemma~\ref{L31}.
\end{proof}

Clearly, a positive lower bound on $\lambda_0$ can be obtained by optimizing its choice according to \eqref{r0}.

\medskip

To finish off the proof of Theorem~\ref{A} it remains to improve the regularity of $(u,\phi_u)$ and to pull it back on the domain $\Omega(u)$ by means of the transformation $T_u$ from \eqref{Tu}.

\begin{cor}\label{c1}
If $(u,{\phi_u})$ is the solution to \eqref{23}-\eqref{34} for $\lambda\in (0,\lambda_0]$ provided by Proposition~\ref{P1},  then  $(u,\psi)$ with $\psi:= \phi_u\circ T_u$ is a solution to \eqref{1}-\eqref{4} with regularity
$$
u\in C^{2+\alpha}\big([-1,1]\big)\ ,\quad  \psi\in W_2^2\big(\Omega(u)\big)\cap C\big(\overline{\Omega(u)}\big)\cap C^{2+\alpha}\big(\Omega(u)\cup\Gamma(u)\big)\ ,
$$
for each $\alpha\in [0,1)$.
\end{cor}

\begin{proof}
From $\phi_u\in W_2^2(\Omega)$ and \eqref{Tu} we readily deduce $\psi=\phi_u\circ T_u\in W_2^2\big(\Omega(u)\big)$ and solves \eqref{1} in $\Omega(u)$. Moreover, since $\Omega(u)$ is a Lipschitz domain, the trace of $\psi$ is well defined as an element of $W_2^{1/2}\big(\partial\Omega(u)\big)$ according to \cite[Chapt.~2, Thm.~5.5]{Necas67} and \eqref{2} follows from $u(\pm1)=0$ and \eqref{24}. Also, since $\Omega$ satisfies the exterior cone condition at every point of its boundary and $u\in W_\infty^2(-1,1)$, it follows from \cite[Thm.~9.30]{GilbargTrudinger} that $\phi_u\in C(\bar{\Omega})$. Recalling that $T_u\in C\left( \overline{\Omega(u)} ; \bar{\Omega} \right)$, we deduce that $\psi\in C\left( \overline{\Omega(u)} \right)$. Finally, $\psi$ is even in $x$ due to Lemma~\ref{L32} and the fact that $u$ is even.

We next improve the regularity of $\psi$ with the help of \cite[Thm.~5.2.7]{Grisvard}. To this end, we note that, since $u\in W_\infty^2(-1,1)$, the boundary $\partial\Omega(u)$ of $\Omega(u)$ is a curvilinear polygon of class $C^{1,1}$ in the sense of \cite[Definition~1.4.5.1]{Grisvard} with four vertices $\{(-1,-1), (1,-1), (-1,0), (1,0) \}$ connected by $W_\infty^2$-smooth curves. In order to apply \cite[Thm.~5.2.7]{Grisvard}, we have to study more precisely the behaviour of the operator $\e^2\partial_x^2 + \partial_z^2$ at these four vertices. Actually, since the operator $\e^2\partial_x^2 + \partial_z^2$ coincides with its principal part and has constant coefficients, we only have to compute the measure $\omega_V$ of the angle at each vertex $V$ of $\Omega(u)$. Obviously, $\omega_{(\pm 1, -1)} = \pi/2$ while 
$$
\omega_{(\pm 1, 0)} = \arccos\left( \frac{\left( \partial_x u(\pm 1) \right)^2}{1 + \left( \partial_x u(\pm 1) \right)^2} \right) \in \left( 0 , \frac{\pi}{2} \right)\,.
$$
Since $\omega_V\in (0,\pi/2]$ for $V\in \{(-1,-1), (1,-1), (-1,0), (1,0) \}$, it follows from \cite[Thms.~5.2.2 and~5.2.7]{Grisvard} that no singularity occurs at the vertices and that $\psi\in W_p^2(\Omega(u))$ for all $p\in (2,\infty)$. The classical Sobolev embedding then implies that $\psi\in C^{1+\alpha}(\overline{\Omega(u)})$ for all $\alpha\in (0,1)$. Combining this regularity with that of $u$ gives that $x\mapsto \e^2 \left| \partial_x \psi(x,u(x)) \right|^2 + \left| \partial_z \psi(x,u(x)) \right|^2$ belongs to $C^\alpha([-1,1])$ and Schauder estimates applied to \eqref{33} guarantee that $u\in C^{2+\alpha}([-1,1])$ for $\alpha\in (0,1)$.

Furthermore, since $u\in C^{2+\alpha}([-1,1])$, it is easy to check that $\Omega(u)$ satisfies an exterior sphere condition at each boundary point $x\in\partial\Omega(u)$. So \cite[Thm.~6.13]{GilbargTrudinger} applied to \eqref{1}-\eqref{2} yields $\psi\in C\big(\overline{\Omega(u)}\big)\cap C^{2+\alpha}\big(\Omega(u)\big)$. Finally, as $\Gamma(u)$ surely  is a $C^{2+\alpha}$ boundary portion of  $\partial\Omega(u)$, we may invoke \cite[Lem.~6.18]{GilbargTrudinger} to deduce that $\psi\in C^{2+\alpha}\big(\Omega(u)\cup \Gamma(u)\big)$.
\end{proof}

The proof of Theorem~\ref{A} is thus complete. 

\section{The vanishing aspect ratio limit: Proof of Theorem~\ref{B}}\label{Sec3} 

We now prove Theorem~\ref{B} and thus consider a family of solutions $(u_\varepsilon,\psi_\varepsilon)_{\varepsilon\in (0,1)}$  to \eqref{1}-\eqref{4} satisfying the bounds \eqref{41} and \eqref{42} for some fixed $\lambda>0$. 

For $\e\in (0,1)$, we set 
$$
\phi_\e := \phi_{u_\varepsilon}=\psi_\e \circ T_{\ue}^{-1}
$$ 
with $T_{\ue}^{-1}$ from \eqref{Tuu} and 
$$
\Phi_\e(x,\eta) := \phi_\e(x,\eta)-\eta\ ,\quad (x,\eta)\in\bar{\Omega}\ .
$$ 
We first derive estimates on $\Phie$ which are uniform with respect to $\e\in (0,1)$. In the following, $K$ denotes an arbitrary positive constant depending only on $\lambda$ and $\kappa_0$.

\begin{lem}\label{le4.1}
There exists a positive constant $K_1$ depending only on $\lambda$ and $\kappa_0$ such that, for $\e\in (0,1)$, 
\begin{eqnarray}
\|\Phie\|_{L_\infty(\Omega)} & \le & 1\,, \label{4.3a} \\
\|\Phie\|_{L_2(\Omega)} & \le & K_1\ \sqrt{\e}\,, \label{4.4} \\
\|\partial_\eta \Phie\|_{L_2(\Omega)} & \le & K_1\ \varepsilon\,, \label{4.3} \\
\|\partial_\eta^2 \Phie\|_{L_2(\Omega)} & \le & K_1\ \varepsilon^2\,. \label{4.5}
\end{eqnarray}
\end{lem}

\begin{proof}
Since $0\le \phi_\e \le 1$ by \eqref{266}, we readily obtain \eqref{4.3a}. We next introduce 
$$
f_\e(x,\eta) := f_{u_\e}(x,\eta) = \e^2\ \eta\ \left( 2 \left( \frac{\partial_x u_\e(x)}{1+u_\e(x)} \right)^2 - \frac{\partial_x^2 u_\e(x)}{1+u_\e(x)} \right)\,, \qquad (x,\eta)\in \bar{\Omega}\,, 
$$
and observe that \eqref{41} and \eqref{42} ensure that
\begin{equation}
\|f_\e\|_{L_\infty(\Omega)} \le \left( \frac{2\e^2}{\kappa_0^4} + \frac{\e^2}{\kappa_0^2} \right)\,. \label{spirou}
\end{equation}
Now, it follows from \eqref{23a}-\eqref{24a} that 
\begin{equation*}
\begin{split}
\int_\Omega f_\e\ \Phie\ \rd (x,\eta) & =\,  \e^2\ \int_\Omega \left[ |\partial_x\Phie|^2 - 2\eta\ \frac{\partial_x u_\e}{1+u_\e}\ \partial_x\Phie\ \partial_\eta\Phie - 2\eta\ \partial_x \left( \frac{\partial_x u_\e}{1+u_\e} \right)\ \Phie\ \partial_\eta\Phie \right]\ \rd (x,\eta) \\
& \quad\, +  \int_\Omega \left[ \frac{1+\e^2\eta^2\ (\partial_x u_\e)^2}{(1+u_\e)^2}\ |\partial_\eta\Phie|^2 + 2\eta\e^2\ \left( \frac{\partial_x u_\e}{1+u_\e} \right)^2\ \Phie\ \partial_\eta\Phie \right]\ \rd (x,\eta) \\
& \quad\, -  \e^2\ \int_\Omega \eta\ \left( 2 \left( \frac{\partial_x u_\e}{1+u_\e} \right)^2 - \frac{\partial_x^2 u_\e}{1+u_\e} \right)\ \Phie\ \partial_\eta\Phie\ \rd (x,\eta) \\
& =\,  \e^2\ \int_\Omega \left( \partial_x\Phie - \eta\ \frac{\partial_x u_\e}{1+u_\e}\ \partial_\eta\Phie \right)^2\ \rd (x,\eta) + \int_\Omega \frac{|\partial_\eta\Phie|^2}{(1+u_\e)^2}\ \rd (x,\eta) \\
& \quad\, +  \e^2\ \int_\Omega \eta\ \left( 2 \left( \frac{\partial_x u_\e}{1+u_\e} \right)^2 - \frac{\partial_x^2 u_\e}{1+u_\e} \right)\ \Phie\ \partial_\eta\Phie\ \rd (x,\eta)\,.
\end{split}
\end{equation*}
We deduce from \eqref{41}, \eqref{42}, \eqref{4.3a}, and the above identity that
\begin{eqnarray*}
\int_\Omega f_\e\ \Phie\ \rd (x,\eta) & \ge & \|\partial_\eta\Phie\|_{L_2(\Omega)}^2 - \left( \frac{2\e^2\ |\Omega|^{1/2}}{\kappa_0^4} + \frac{\e^2\ |\Omega|^{1/2}}{\kappa_0^2} \right)\ \|\partial_\eta\Phie\|_{L_2(\Omega)} \\
& \ge & (1-\e^2)\ \|\partial_\eta\Phie\|_{L_2(\Omega)}^2 - K\ \e^2\,,
\end{eqnarray*}
while \eqref{4.3a} and \eqref{spirou} ensure that
$$
\int_\Omega f_\e\ \Phie\ \rd (x,\eta) \le \left( \frac{2\e^2}{\kappa_0^4} + \frac{\e^2}{\kappa_0^2} \right)\ |\Omega|\,.
$$
Combining the above two inequalities gives \eqref{4.3}. Next, thanks to \eqref{24a} and \eqref{4.3a}, we have
\begin{eqnarray*}
\int_\Omega |\Phie(x,\eta)|^2\ \rd (x,\eta) & = & \int_\Omega \left| \int_\eta^1 \partial_\eta \Phie(x,y)\ \rd y \right|^2\ \rd (x,\eta) \le \int_\Omega \int_0^1 |\partial_\eta\Phie(x,y)|^2\ \rd y\, \rd (x,\eta) \\
& \le & |\Omega|\ \|\partial_\eta\Phie\|_{L_2(\Omega)}^2\,,
\end{eqnarray*}
and \eqref{4.4} readily follows from the previous inequality and \eqref{4.3}.

Finally, setting $\zeta_\e:=\partial_\eta^2\Phie$ and \mbox{$\omega_\e:=\partial_x\partial_\eta\Phie$}, we infer from \eqref{23a}-\eqref{24a} that 
$$
\int_\Omega \left[ \frac{1+\e^2\eta^2\ (\partial_x u_\e)^2}{(1+u_\e)^2}\ \zeta_\e^2 + \e^2\ \partial_x^2\Phie\ \zeta_\e - 2\eta\e^2\ \frac{\partial_x u_\e}{1+u_\e}\ \zeta_\e\ \omega_\e \right]\ \rd (x,\eta) = \int_\Omega f_\e\ \left( 1 - \partial_\eta\Phie \right)\ \zeta_\e\ \rd (x,\eta)\,.
$$
Since
$$
\int_\Omega \partial_x^2\Phie\ \zeta_\e\ \rd (x,\eta) = \int_\Omega \omega_\e^2\ \rd (x,\eta)
$$
by {\cite[Lem.~4.3.1.2 $\&$~4.3.1.3]{Grisvard}}, the above identity also reads
$$
\int_\Omega \left[ \frac{\zeta_\e^2}{(1+u_\e)^2} + \e^2\ \left( \omega_\e - \eta\ \frac{\partial_x u_\e}{1+u_\e}\ \zeta_\e \right)^2 \right]\ \rd (x,\eta) = \int_\Omega f_\e\ \left( 1 - \partial_\eta\Phie \right)\ \zeta_\e\ \rd (x,\eta)\,.
$$
Consequently, owing to \eqref{41}, \eqref{4.3} and \eqref{spirou}, we deduce from the above identity that
\begin{eqnarray*}
\|\zeta_\e\|_{L_2(\Omega)}^2 & \le & \int_\Omega \frac{\zeta_\e^2}{(1+u_\e)^2}\ \rd (x,\eta) \le \|f_\e\|_{L_\infty(\Omega)}\ \left( |\Omega|^{1/2} + \|\partial_\eta\Phie\|_{L_2(\Omega)} \right)\ \|\zeta_\e\|_{L_2(\Omega)} \\
& \le & K\ \e^2\ \|\zeta_\e\|_{L_2(\Omega)}\,,
\end{eqnarray*} 
whence \eqref{4.5}.
\end{proof}

In order to pass to the limit as $\e\to 0$ in \eqref{33}, we need to control the trace of $\partial_\eta\Phie$ on $(-1,1)\times \{1\}$. For that purpose, the following lemma will be adequate. 

\begin{lem}\label{le4.2}
Given $\vartheta\in W_2^2(\Omega)$, we have
$$
\|\partial_\eta\vartheta(.,1)\|_{L_2(-1,1)} \le \sqrt{2}\ \left( \|\partial_\eta\vartheta\|_{L_2(\Omega)}  + \|\partial_\eta^2\vartheta\|_{L_2(\Omega)} \right)\,.
$$
\end{lem}

\begin{proof} We first assume that $\vartheta\in C^\infty(\bar{\Omega})$. For $x\in (-1,1)$ and $\eta\in (0,1)$, we have
$$
\partial_\eta\vartheta(x,1) = \partial_\eta\vartheta(x,\eta) + \int_\eta^1 \partial_\eta^2\vartheta(x,y)\ \rd y\,.
$$
Integrating this identity with respect to $\eta$ over $(0,1)$ gives
$$
\partial_\eta\vartheta(x,1) = \int_0^1 \partial_\eta\vartheta(x,\eta)\ \rd \eta + \int_0^1 \int_\eta^1 \partial_\eta^2\vartheta(x,y)\ \rd y\, \rd \eta\,,
$$
and thus
$$
\left| \partial_\eta\vartheta(x,1) \right| \le \|\partial_\eta\vartheta(x,.)\|_{L_2(0,1)} + \|\partial_\eta^2\vartheta(x,.)\|_{L_2(0,1)}\,.
$$
Raising both sides of the above identity to the square and integrating with respect to $x$ over $(-1,1)$ gives the claimed estimate for smooth functions. The general case follows by a density argument, see, e.g.,  {\cite[{Thm.~3.18}]{Adams_SobolevSpaces}} or \cite[Chapt.~2, Thm.~3.1]{Necas67}.
\end{proof}

A control on the trace of $\partial_\eta\Phie$ on $(-1,1)\times \{1\}$ follows at once from Lemma~\ref{le4.1} and Lemma~\ref{le4.2}.

\begin{cor}\label{cor4.3}
There exists a positive constant $K_2$ depending only on $\lambda$ and $\kappa_0$ such that, for $\e\in (0,1)$, 
\begin{equation}
\|\partial_\eta\Phie(.,1)\|_{L_2(-1,1)} \le K_2\ \e\,. \label{4.6}
\end{equation}
\end{cor}

\begin{proof}[Proof of Theorem~\ref{B}] By \eqref{42} and the Arzel\`a-Ascoli theorem, there are a sequence $(\e_k)_{k\ge 1}$ and a function $u_0\in W_\infty^2(-1,1)$ such that $\e_k\to 0$ in $(0,1)$  and
\begin{eqnarray}
u_{\e_k} \longrightarrow u_0 \quad & \text{ in } & \quad W_\infty^1(-1,1)\,, \label{4.7} \\
u_{\e_k} \stackrel{*}{\rightharpoonup} u_0 \quad & \text{ in } & \quad W_\infty^2(-1,1)\,. \label{4.8}
\end{eqnarray}
Owing to \eqref{41}, \eqref{4.6}, \eqref{4.7}, and the definition of $\Phie$, we have
\begin{equation}
0\ge u_0(x)\ge \kappa_0-1\,, \qquad x\in [-1,1]\,, \label{4.9a}
\end{equation}
and
\begin{equation}
\frac{1 + \e_k^2\ \left( \partial_x u_{\e_k} \right)^2}{\left( 1+u_{\e_k} \right)^2}\ \left| \partial_\eta\phi_{\e_k}(.,1) \right|^2 \longrightarrow \frac{1}{(1+u_0)^2} \quad\text{ in }\quad L_1(-1,1)\,. \label{4.9}
\end{equation}
Combining \eqref{4.8} and \eqref{4.9}, we may pass to the limit as $\e_k\to 0$ in \eqref{33} and conclude that $u_0$ is a solution to the small aspect ratio equation \eqref{uu}, where  $u_0(\pm 1)=0$ is guaranteed by \eqref{34} and \eqref{4.7}. 

Also, by \eqref{4.4},
$$
\lim_{k\to\infty} \int_\Omega \left| \phi_{\e_k}(x,\eta) -\eta \right|^2\ \rd (x,\eta) = 0\,,
$$
and, since
\begin{eqnarray*}
\int_\Omega \left| \phi_{\e_k}(x,\eta) -\eta \right|^2\ \rd (x,\eta) & = & \int_{-1}^1 \int_{-1}^{u_{\e_k}(x)} \left| \psi_{\e_k}(x,z) - \frac{1+z}{1+u_{\e_k}(x)}\right|^2\ \frac{\rd z\,\rd x}{1+u_{\e_k}(x)} \\
& \ge &  \int_{-1}^1 \int_{-1}^{u_{\e_k}(x)} \left| \psi_{\e_k}(x,z) - \frac{1+z}{1+u_{\e_k}(x)}\right|^2\ \rd z\,\rd x\,,
\end{eqnarray*}
we readily obtain \eqref{411}, where $\psi_0$ is given by \eqref{sar} with $u=u_0$.
\end{proof}

\section*{Acknowledgments}

{We thank Matthieu Hillairet} for helpful discussions. This work was performed while Ch.W. was a visitor at the Institut de Math\'{e}matiques de Toulouse, Universit\'{e} Paul Sabatier. The kind hospitality is gratefully acknowledged.


\bibliographystyle{abbrv}
\bibliography{MEMS}

\end{document}